\documentclass[12pt]{amsart}

\usepackage{amssymb,mathrsfs,bm}

\usepackage{enumerate}

\usepackage[centering]{geometry}
\usepackage[pagebackref]{hyperref}

\renewcommand*\backref[1]{}
\renewcommand*\backrefalt[4]{}% \ifcase #1% No citations.% \or (page #2).% \else (pages #2).% \fi 

\theoremstyle{plain}
\newtheorem{thm}{Theorem}[section]
\newtheorem*{thm*}{Theorem}
\newtheorem{cor}[thm]{Corollary}
\newtheorem{lem}[thm]{Lemma}

\theoremstyle{definition}

\newtheorem{rem}{Remark}[section]
\numberwithin{equation}{section}

\newtheoremstyle{citing}% name
{3pt}%      Space above, empty = `usual value'
{3pt}%      Space below
{\itshape}% Body font
{}%         Indent amount (empty = no indent, \parindent = para indent)
{\bfseries}% Thm head font
{.}%        Punctuation after thm head
{.5em}%     Space after thm head: " " = normal interword space;
{\thmnote{#3}}% Thm head spec
\theoremstyle{citing}

\DeclareMathOperator{\Hm}{\mathcal{H}}

\DeclareMathOperator{\diam}{diam}

\begin{document}
	
\title{Random covering sets in metric space with exponentially mixing property}

\author{Zhang-nan Hu}
\author{Bing Li*}\thanks{* Corresponding author}
\address{School of Mathematics, South China University of Technology, Guangzhou, China, 510641}
\email{hnlgdxhzn@163.com, scbingli@scut.edu.cn}

\begin{abstract}
	Let $\{B(\xi_n,r_n)\}_{n\ge1}$ be a sequence of random balls whose centers $\{\xi_n\}_{n\ge1}$ is a stationary process, and $\{r_n\}_{n\ge1}$ is a sequence of positive numbers decreasing to 0. Our object is the random covering set $E=\limsup\limits_{n\to\infty}B(\xi_n,r_n)$, that is, the points covered by $B(\xi_n,r_n)$ infinitely often. The sizes of $E$ are investigated from the viewpoint of measure, dimension and topology.
\end{abstract}


\maketitle
	
\section{Introduction}

Let $(X,d)$ be a complete metric space. Given a sequence of points $\{x_n\}_{n\ge1}$ in $X$, let $\{r_n\}_{n\ge1}$ be a sequence of positive numbers decreasing to 0. A general covering problem concerns the sets 
\[\limsup\limits_{n\to\infty}B(x_n,r_n)=\bigcap_{k=1}^{\infty}\bigcup_{n=k}^{\infty}B(x_n,r_n),\]
where $B(x_n,r_n)$ denotes the ball centered at $x_n$ of radius $r_n$. Many authors have investigated the size and structure of these limsup sets. 

One of classical models is to let $\{x_n\}_{n\ge1}$ be a sequence of random variables on $(\Omega,\mathcal{B},\mathbb{P})$. The limsup set $E=\bigcap_{k=1}^{\infty}\bigcup_{n=k}^{\infty}B(x_n,r_n)$ is usually called random covering set, since it consists of the points covered by random balls $\{B(x_n,r_n)\}$ infinitely many times. The study of random covering sets has a long and convoluted history~\cite{Durand10,Eks19,EJS18,Fan02,LSX13,LS14,Per15}. 

In 1956, Dvoretzky \cite{Dvor56} called the attention on the study of such random covering sets in the circle $\mathbb{T}$ with $\{x_n\}_{n\ge1}$ being independent and uniformly distributed. He asked the question when $E=\mathbb{T}$ a.s. or not. There was a series of contributions.  In 1971, L. Shepp~\cite{Shepp72} gave a sufficient and necessary condition: $E=\mathbb{T}$ a.s. if and only if $\sum_{n=1}^{\infty}(1/n^2)\exp(r_1+\dots+r_n)=\infty$. Kahane \cite{Kahane85} proved that $E$ is a.s. dense on $\mathbb{T}$ and moreover of second category. The applications of the Borel-Cantelli lemma and Fubini's theorem give that the Lebesgue measure of $E$ is 0 or 1 a.s. according to the convergence or divergence of the series $\sum_{n=1}^{\infty}r_n$. Many authors have studied the Hausdorff dimension and other fractal properties of random covering set $E$. Fan and Wu \cite{FanW04} considered the special case $r_n=a/{n^{\alpha}}$ with $a>0$ and $\alpha>1$. They proved that $\dim_{\rm{H}} (E)=1/{\alpha}$ a.s., where $\dim_{\rm{H}}$ denotes Hausdorff dimension. Durand \cite{Durand10} considered a general sequence $\{r_n\}_{n\ge1}$ and proved $\dim_{\rm{H}}E=\inf\{s>0\colon \sum_{n=1}^{\infty}r_n^s<\infty\}$ and $\dim_{\rm{P}} E=1$ a.s., where $\dim_{\rm{P}}$ denotes packing dimension. 

Many variations of the random covering problems have been addressed by many mathematicians. For example, J\"arvenp\"a\"a, $et~al.$ \cite{JKLS14} covered the torus by self-affine sets instead of balls. Feng, $et~al.$ \cite{FJV18} extended it to any open sets. The covering model in Ahlfors regular metric space has been studied in \cite{JKLSX17}.

Instead of a sequence of random variables, many authors considered the case that $\{x_n\}_{n\ge1}$ is the orbit of a dynamical system $(X,T)$. Define the dynamical covering set $E(x,r_n)=\bigcap_{k=1}^{\infty}\bigcup_{n=k}^{\infty}B(T^{n-1}x,r_n).$
Fan, Schmeling and Troubetzkoy~\cite{FanScT13} computed the Hausdorff dimension of $E(x,r_n)$ when $X$ is the unit interval and $T\colon x\mapsto 2x$ $(\bmod\thinspace1)$. In 2013, Liao and Seuret~\cite{LiaSe13} successfully enlarged the setting to finite expanding Markov map with Gibbs measure $m$. They proved that if $r_n=n^{-\alpha}$, then $\dim_{\rm{H}} E(x,r_n)=1/\alpha$ a.e. provided that $1/\alpha$ is not larger than the dimension of the measure $m$. Pesson and Rams \cite{PR17} considered more general piecewise expanding maps than Markov maps. In 2017, Wang, Wu and Xu~\cite{WanWuX17} considered the dynamical covering problem on the triadic cantor set.

In this paper, we consider the covering set with $\{\xi_n\}_{n\ge1}$ which is a sequence of points in a compact metric space $(X,d)$, chosen randomly. And independence of $\{\xi_n\}_{n\ge1}$ is not necessary. The purpose of this article is to study some properties of random covering set in general probabilistic setting, including measure, density, fractal dimensions and so on. Next we will state the main results and provide some discussions. Section 2 is devoted to the proof of the results. In the last section 3, we will give an application to the dynamical covering problem.

Let $\{\xi_n\}_{n\ge1}$ be a stationary process on a probability space $(\Omega,\mathcal{B},\mathbb{P})$ and take values in a compact metric space $(X,d)$. Let $\mu$ be the probability measure defined by 
\begin{equation}\label{eqmu}
\mu(A)=\mathbb{P}(\xi_1\in A)
\end{equation}
for any Borel set $A \subset X$. Assume that $X$ is the support of $\mu$.

We say that $\{\xi_n\}_{n\ge1}$ is $exponentially~ mixing$ if for any $n\ge1$, there exist two constants $c>0$ and $0<\gamma<1$ such that  
\[\big|\mathbb{P}(\xi_1\in A|D)-\mathbb{P}(\xi_1\in A)\big|\le c\gamma^{n}\]
holds for any ball $A\subset X$ and $D\in\mathcal{B}^{n+1}$, where $\mathcal{B}^{n+1}$ is the sub-$\sigma$-field generated by $\{\xi_{n+i}\}_{i\ge1}$.

Recall that a Borel measure $\mu$ is $\textit{Ahlfors}$ $s$-$\textit{regular}$ $(0< s < \infty)$ if there exists a constant $0 < C <\infty$ such that 
\begin{equation}\label{eqah}
C^{-1}r^s\le\mu\big(B(x,r)\big)\le Cr^s
\end{equation} 
holds for all $x\in X$ and $0<r\le\diam X$, where $\diam X$ is the diameter of X. 

A metric space $X$ is said to be $\textit{Ahlfors}$ $s$-$\textit{regular}$ if there exists a Borel measure on $X$ satisfying formula \eqref{eqah}.

Let $\{r_n\}_{n\ge1}$ be a sequence of positive real numbers decreasing to zero. For every $n\ge 1$, denote $B_n:=B(\xi_n,r_n)$. Define 
\[E:=\limsup_{n\to\infty}B_n=\{y\in X\colon y\in B_n ~\text{for~ infinitely ~many}~n\ge1\}.\]
The set $E$ is a $random ~covering~ set$ and consists of the points which are covered by $\{B_n\}_{n\ge1}$ infinitely often (i.o. for short).

\begin{thm}\label{thm1}
	Let $\{\xi_n\}_{n\ge1}$ be exponentially mixing and the probability measure $\mu$ defined in~\eqref{eqmu} be Ahlfors s-regular. Then we have
	\[\mu\bigl(E\bigr)=
	\begin{cases}
	0&  \text{if $\sum_{n=1}^{\infty}r_n^s<\infty$}\\
	1&  \text{if $\sum_{n=1}^{\infty}r_n^s=\infty$}
	\end{cases} \quad a.s.\]
\end{thm}

A dimension function $f\colon \mathbb{R}^{+}\to\mathbb{R}^{+}$ is a continuous and non-decreasing function such that $f(r)\to0$ as $r\to0$. If there exists a constant $\eta>1$ such that for $r>0$, $f(2r)\le\eta f(r)$, then we say that function $f$ is $doubling$.

\begin{thm}\label{thm2}
	Let $\{\xi_n\}_{n\ge1}$ be exponentially mixing and the probability measure $\mu$ defined in~\eqref{eqmu} be Ahlfors s-regular. Suppose that $f$ is a doubling dimension function with $f(r)/r^s$ being nondecreasing as $r\to0$. Then, with probability one, for any ball $B$ of $X$,
\[\Hm^f(E\cap B)=
\begin{cases}
	0&  \text{if $\sum_{n=1}^{\infty}f(r_n)<\infty$}\\
	\Hm^f(B)&  \text{if $\sum_{n=1}^{\infty}f(r_n)=\infty$}
	\end{cases} \quad. \]
	Futhermore,
	\[\dim_{\rm{H}} E=\alpha \quad a.s.,\]
	where $\alpha=\inf\{t\le s\colon \sum_{n=1}^{\infty}r_n^t<\infty\}$.
\end{thm}

\begin{thm}\label{thm3}
	Let $\{\xi_n\}_{n\ge1}$ be exponentially mixing and the probability measure $\mu$ be Ahlfors~ s-regular. Then random covering set $E$ is dense in X almost surely. 
\end{thm}

\begin{cor}\label{cor1}
	Assume that the conditions of Theorem~\ref{thm2} hold. We have $\dim_{\rm{B}} E=s$ almost surely, where $\dim_{\rm{B}}$ denotes box dimension.
\end{cor}

Recall that a set is called $residual$ if the complement of the set is a first category set.

\begin{thm}\label{thm4}
	Let $\{\xi_n\}_{n\ge1}$ be exponentially mixing and the probability measure $\mu$ be Ahlfors~ s-regular. Then random covering set $E$ is a residual set almost surely. And $E$ is also a set of second category almost surely. In particular $\dim_{\rm{P}} E=s$ almost surely.
\end{thm}

\begin{rem}\label{rem1}
         ~
	\begin{enumerate}[(i)]
        \item J. Heinonen~\cite{Hein01} proved that, if $X$ is a metric space admitting a Borel measure $\mu$ which is $\textit{Ahlfors}$ $s$-$\textit{regular}$ $(0< s < \infty)$, then $X$ has Hausdorff dimension precisely $s$.
        \item From the proof of Corollary~\ref{cor1}, we see that the box dimension of the space we considered is $s$.
        \item By Cutler~\cite[Theorem 3.16]{Cutler95}, we derive that the packing dimension of the space we considered is $s$.
        \end{enumerate}
\end{rem}

\section{Proofs of the main results}

In this section we proved Theorems~\ref{thm1}~--~\ref{thm4} and Corollary~\ref{cor1}.

\begin{lem}\label{lem3}
	Suppose that $\{\xi_n\}_{n\ge1}$ is an exponentially mixing stationary process. Let $\{h_n\}_{n\ge1}$ be a decreasing sequence. For any point $y\in X$, if the series $\sum_{n=1}^{\infty}h_n^s$ diverges, we have $\mathbb{P}(\xi_n\in B(y,h_n)~i.o.)=1$.
	
\end{lem}

\begin{proof}
         Let $y\in X$ and denote $\widetilde{J}_n=\{\omega\in\Omega\colon\xi_n(\omega)\in B(y,h_n)\}$. Let $N\ge1$ and $S_N=\sum_{n=1}^{N}\chi_{\widetilde{J}_n}$, where $\chi$ is the indicator function. Then
	\[\mathbb{E}(S_N)=\sum_{n=1}^N\mathbb{P}(\widetilde{J}_n)=\sum_{n=1}^N\mu\bigl(B(y,h_n)\bigr).\]
	Since $\mu$ is Ahlfors $s$-regular, by (\ref{eqah}), we have
	\[\mathbb{E}(S_N)\ge C^{-1}\sum_{n=1}^N h_n^s\to\infty,\quad \text{as} ~N\to\infty.\]
	 Thus $\sum_{n=1}^{\infty}\mathbb{P}(\widetilde{J}_n)=\lim\limits_{N\to\infty}\mathbb{E}(S_N)=\infty$.

	 By the Paley-Zygmund inequality, for all $0<\lambda<1$, we have
	\begin{equation}\label{eq21}
	\begin{split}
	\mathbb{P}\left(S_N\ge\lambda\mathbb{E}(S_N)\right)&\ge(1-\lambda)^2\frac{{\mathbb{E}}^2(S_N)}{\mathbb{E}(S_N^2)}\\
	&=(1-\lambda)^2\frac{\left(\sum_{n=1}^N{\mathbb{P}(\widetilde{J}_n)}\right)^2}{\mathbb{E}(S_N^2)}.
	\end{split}
	\end{equation}
	Now we estimate $\mathbb{E}(S_N^2)$,
	\begin{equation}\label{eq22}
	\begin{split}
	\mathbb{E}(S_N^2)&=\mathbb{E} \left(\sum_{n=1}^N \chi_{\widetilde{J}_n}\right)^2=\mathbb{E}\bigl(\sum_{n=1}^N \chi_{\widetilde{J}_n}+\sum_{n=1}^N\sum_{\substack{m=1\\ m\ne n}}^N \chi_{\widetilde{J}_n\cap\widetilde{J}_m}\bigr)\\
	&=\sum_{n=1}^N \mathbb{P}(\widetilde{J}_n)+2\sum_{n=1}^N\sum_{m=1}^{n-1}\mathbb{P}(\widetilde{J}_n\cap \widetilde{J}_m).
	\end{split}
	\end{equation}
	Using the stationarity and the exponentially mixing property of the process $\{\xi_n\}$, we get 
	\[\begin{split}
	\mathbb{P}(\widetilde{J}_n\cap\widetilde{J}_m)&=\mathbb{P}\{\xi_n\in B(y,h_n),\xi_m\in B(y,h_m)\}=\mathbb{P}\{\xi_1\in B(y,h_m),\xi_{n-m+1}\in B(y,h_n)\}\\
	&\le\mathbb{P}\bigl(\xi_1\in B(y,h_m)\bigr)\mathbb{P}\bigl(\xi_{n-m+1}\in B(y,h_n)\bigr)+c{\gamma}^{n-m}\mathbb{P}\bigl(\xi_{n-m+1}\in B(y,h_n)\bigr)\\
	&=\mathbb{P}(\widetilde{J}_n)\mathbb{P}(\widetilde{J}_m)+c{\gamma}^{n-m}\mathbb{P}(\widetilde{J}_n).
	\end{split}\]
	Hence the equality \eqref{eq22} reads as follows
	\begin{equation}\label{eq23}
	\begin{split}
	\mathbb{E}(S_N^2)&\le\sum_{n=1}^N \mathbb{P}(\widetilde{J}_n)+2\sum_{n=1}^N\sum_{m=1}^{n-1}\left(\mathbb{P}(\widetilde{J}_n)\mathbb{P}(\widetilde{J}_m)+c{\gamma}^{n-m}\mathbb{P}(\widetilde{J}_n)\right)\\
	&\le\sum_{n=1}^N \mathbb{P}(\widetilde{J}_n)+\left(\sum_{n=1}^N\mathbb{P}(\widetilde{J}_n)\right)^2+\frac{2c\gamma}{1-\gamma}\sum_{n=1}^N \mathbb{P}(\widetilde{J}_n)\\
	&\le c'\sum_{n=1}^N\mathbb{P}(\widetilde{J}_n)+\left(\sum_{n=1}^N\mathbb{P}(\widetilde{J}_n)\right)^2,
	\end{split}
	\end{equation}
	where $c'$ is a constant. Combining~~\eqref{eq21} and~~\eqref{eq23}, we derive that
	\begin{equation}\label{eq24}
	         \mathbb{P}\left(S_N\ge\lambda\mathbb{E}(S_N)\right)\ge(1-\lambda)^2\frac{\left(\sum_{n=1}^N\mathbb{P}(\widetilde{J}_n)\right)^2}{c'\sum_{n=1}^N\mathbb{P}(\widetilde{J}_n)+\left(\sum_{n=1}^N\mathbb{P}(\widetilde{J}_n)\right)^2}\to1,
	 \end{equation}
	as $N\to\infty$ and $\lambda\to0$ due to the divergence of the series $\sum_{n=1}^{\infty}\mathbb{P}(\widetilde{J}_n)$.
	We notice that 
	\[\{\omega\in\widetilde{J}_n~i.o.\}=\{\lim\limits_{N\to\infty}S_N=\infty\}\supset\{S_N\ge\lambda\mathbb{E}(S_N)\}.\]
	By \eqref{eq24} we have 
	\[\mathbb{P}(\widetilde{J}~i.o.)\ge\lim\limits_{N\to\infty}\mathbb{P}\bigl(S_N\ge\lambda\mathbb{E}(S_N)\bigr)=1,\]
	which derives $\mathbb{P}(\widetilde{J}~i.o.)=1$. 
\end{proof}

\begin{proof}[Proof of Theorem~\ref{thm1}]
	First we show that $\mu\bigl(E\bigr)=0$ for any $\omega\in\Omega$ if $\sum_{n=1}^{\infty}r_n^s<\infty$.
	
	Since $\mu$ is Ahlfors $s$-regular, by~\eqref{eqah}, we have $\sum_{n=1}^{\infty}\mu(B_n)\le C\sum_{n=1}^{\infty}r_n^s<{\infty}$. By the Borel-Cantelli lemma, $\mu(\limsup\limits_{n\to\infty} B_n)=0$, that is $\mu\bigl(E\bigr)=0$.
	
	Now we consider the divergence case. Let $y\in X$ and  
	\[F(y)=\{\omega\in \Omega\colon \xi_n(\omega)\in B(y,r_n)~i.o.\}.\]
	By lemma~\ref{lem3}, we have $\mathbb{P}\bigl(F(y)\bigr)=1$. Since $y\in E\Leftrightarrow \omega\in F(y)$, applying Fubini's theorem gives
	\[
	\begin{split}
	\mathbb{P}\left(\mu(E)\right)&=\int\int1_{E}(y)d\mu(y)d\mathbb{P}(\omega)=\int\int1_{F(y)}(\omega)d\mathbb{P}(\omega)d\mu(y)\\
	&=\int\mathbb{P}\bigl(F(y)\bigr)d\mu(y)=1.
	\end{split}
	\]
	Hence, for $\mathbb{P}$-almost all $\omega$, $\mu\bigl(E\bigr)=1$.
\end{proof}
		
	We recall a general case of the mass transference principle~\cite[Theorem 3]{BerVel06}, suitable for the proof of our Theorem~\ref{thm2}. 
	
	Let $(Y,\rho)$ be a locally compact metric space. Let $g$ be a doubling dimension function and suppose
	there exist  constants $0<c_1<1<c_2<\infty$ and $r_0>0$ such that 
         \[c_1g\bigl(r(B)\bigr)\le\Hm^g(B)\le c_2g\bigl(r(B)\bigr),\]
         for any ball $B=B(x,r)$ with $x\in Y$ and $r\le r_0$. Next, given a dimension function $f$ and a ball $B=B(x,r)$ we define $B^f:=B\bigl(x,g^{-1}f(r)\bigr)$. 

         \begin{thm}[Beresnevich--Velani]\label{thm21}
	Let $(Y,\rho)$ and $g$ be as above and let $\{B_i\}_{i\in\mathbb{N}}$ be a sequence of balls in $Y$ with $r(B_i)\to0$ as $i\to\infty$. Let $f$ be a dimension function such that $f(r)/g(r)$ is monotonic and suppose that for any
	ball $B$ in $Y$
         \[\Hm^g(B\cap \limsup\limits_{i\to \infty}B_i^f)=\Hm^g(B).\]
         Then, for any ball $B$ in $Y$
         \[\Hm^f(B\cap \limsup\limits_{i\to \infty}B_i^g)=\Hm^f(B).\]
         \end{thm}

\begin{proof}[Proof of Theorem~\ref{thm2}]
	For any $\delta>0$, there is an integer $n_0\ge1$ such that $0<2r_n<\delta$ for any $n\ge n_0$. Since $f(r)/r^s$ is increasing as $r\to0$ and $E\subset\bigcup_{n=n_0}^{\infty}B_n$, we have
\[\Hm_{\delta}^f(E)\le \sum_{n=n_0}^{\infty}f(2r_n)\le2^s\sum_{n=n_0}^{\infty}f(r_n).\]
If the series $\sum_{n=1}^{\infty}f(r_n)$ converges, then $\sum_{n=n_0}^{\infty}f(r_n)\to0$ as $\delta\to0$ since it deduces that $n_0\to\infty$. Therefore $\Hm^f(E)=\lim\limits_{\delta\to0}\Hm_{\delta}^f(E)=0$ for any $\omega\in\Omega$.

For the divergence part, denote $B_n^f=B\bigl(\xi_n,\left(f(r_n)\right)^{1/s}\bigr)$. By Theorem \ref{thm1}, we obtain
\[\mu\bigl(\limsup_ {n\to\infty}B_n^f\bigr)=1 ~\text{a.s.},\]
since 
\[\sum_{n=1}^{\infty}\left(\left(f(r_n)\right)^{1/s}\right)^s=\sum_{n=1}^{\infty}f(r_n)=\infty.\]
By Theorem \ref{thm21}, we have 
\[\Hm^f(E\cap B)= \Hm^f(B)~\text{a.s.}.\]
\end{proof}

For proving Theorem~\ref{thm3}, \ref{thm4} and Corollary~\ref{cor1} we will use the following elementary lemma, whose proof can be found in \cite{JRM00}.

\begin{lem}\label{lem1}
If the metric space $(X,d)$ is compact, then $X$ is bounded, separated and complete. 
\end{lem}

%\begin{proof}
% The proof can be referred to \cite{JRM00}.
%\end{proof} 

\begin{proof}[Proof of Theorem~\ref{thm3}]
	The compact metric space $(X,d)$ is separated, hence there exists a countable dense subset denoted by $A$. Letting $\mathcal{A}=\{B(a,1/k)|a\in A,k\ge1\}$, now we prove that $\mathcal{A}$ is a countable base of $(X,d)$. 

    For any open subset $U$ in $X$ and $x\in U$, we can find a $\epsilon_x>0$ such that $B(x,\epsilon_x)\subset U$. There exists $k_x\ge1$ with $2/{k_x}<\epsilon_x$. Since $A$ is dense in $X$, we can choose $a_x\in A$ satisfying $d(x,a_x)<1/{k_x}$. Hence $x\in B(a_x,1/{k_x})\subset B(x, \epsilon_x)\subset U $. It yields that $U=\bigcup\limits_{x\in U}\{x\}\subset\bigcup\limits_{x\in U}B(a_x,1/{k_x})\subset U$, that is $U=\bigcup\limits_{x\in U}B(a_x,1/{k_x})$, where $B(a_x,1/{k_x})\in \mathcal{A}$. Therefore $\mathcal{A}$ is a countable base of $(X,d)$, since $\{B(a_x,1/{k_x})|x\in U\}$ is at most countable. 

	For any $B=B(a,r)\in\mathcal{A}$, we will show $\mathbb{P}(E\cap B\ne\varnothing)=1$.
	
	Denote $\widetilde{A}_n=\{\omega\in \Omega\colon B_n(\omega)\cap B\ne\varnothing\}$. We notice that $\widetilde{A}_n=\{\xi_n\in B(a,r+r_n)\}$.~Then we have 
	\[\sum_{n=1}^N (r+r_n)^s\ge \sum_{n=1}^N r^s\to\infty, \quad\text{as}~ N\to\infty.\]
	 Thus the series $\sum_{n=1}^{\infty}(r+r_n)^s$ diverges. Hence $\mathbb{P}(E\cap B\ne\varnothing)=\mathbb{P}(\widetilde{A}_n~i.o.)=1$ due to lemma~\ref{lem3}. 
	
	For convenience, we denote $\mathcal{A}=\{B^i\}_{i\ge1}$. For any $B^i\in\mathcal{A}$, let $A_i=\{\omega| B^i \cap E =\varnothing\}$, then $\mathbb{P}(A_i)=0$. Thus
	\[\mathbb{P}\{\omega|E\cap B^i=\varnothing \text{~for~some}~i\ge1\}=\mathbb{P}\bigl(\bigcup_{i\ge1}A_i\bigr)\le\sum_{i\ge1}\mathbb{P}(A_i)=0.\]
	So there exists a null probability event outside which for all balls $B^i$ in $\mathcal{A}$, we have $E\cap B^i\ne \varnothing$. For any $x\in X$ and open set $U\ni x$, there exist $\{B^{i_n}\}_{n\ge1}\subset \mathcal{A}$ with $U=\bigcup_{n\ge1}B^{i_n}$. Hence \[
	\mathbb{P}\{E\cap U\ne\varnothing\}\ge\mathbb{P}\{\omega|E\cap B^i\ne\varnothing \text{~for~any}~i\ge1\}=1.\]
	 Therefore $E$ is dense in $X$ a.s..	  
\end{proof}

\begin{proof}[Proof of Corollary~\ref{cor1}]
	Since $X$ is bounded, let $\{y_k\}_{k=1}^n$ be a maximal $r$-separated set of $X$ $(r>0)$. Then the balls $B(k)=B(y_k,r/3)$ are disjoint, that is $B(k)\cap B(j)=\varnothing$ for any $ k\ne j$. Thus, by~\eqref{eqah}
	\[C^{-1}3^{-s}nr^s\le \mu\bigl(B(1)\bigr)+\dots+\mu\bigl(B(n)\bigr)\le1.\]
	On the other hand, the balls $\{3B(k)\}^n_{k=1}$ cover $X$ where $3B(k)=B(y_k,r)$, so
	\[1\le\mu\bigl(3B(1)\bigr)+\dots+\mu\bigl(3B(n)\bigr)\le Cnr^s.\]
	Combining these estimates, it follows that $\dim_{\rm{B}} X=s$. From Theorem~\ref{thm2}, we have $\dim_{\rm{B}} E=\dim_{\rm{B}}X=s~a.s.$. 
\end{proof}

\begin{proof}[Proof of Theorem~\ref{thm4}]
	First we show that $E$ is a residual set a.s..
	
	Now we check if $E^{\mathrm{C}}=\bigcup_{N=1}^{\infty}\bigcap_{n=N}^{\infty} B_n^{\mathrm{C}}$ is a first category set almost surely. Denote $E_N^{\mathrm{C}}=\bigcap_{n=N}^{\infty} B_n^{\mathrm{C}}$. We suppose that there exists $N_1\ge1$ such that $E_{N_1}^{\mathrm{C}}$ is not sparse. That means there is a ball $B(y,r)$ with $r>0$ satisfying 
	\[\overline{B}(y,r)\subset \overline{E_{N_1}^{\mathrm{C}}}=E_{N_1}^{\mathrm{C}}=X\backslash{\bigcup_{n=N_1}^{\infty}B_n}.\]
	 Hence $B(y,r)\cap{\bigcup_{n=N_1}^{\infty}B_n}=\varnothing$. However $ \bigcup_{n=N_1}^{\infty}B_n$ is dense in $X$ a.s., since $E\subset\bigcup_{n=N_1}^{\infty}B_n$ is dense in $X$ a.s., which implies $B(y,r)\cap \bigcup_{n=N_1}^{\infty}B_n\ne\varnothing$. Contradiction. Hence $E^{\mathrm{C}}$ is a set of first category almost surely. It derives that $E$ is a residual set a.s..
	
	From Lemma \ref{lem1}, we know the space $X$ is complete. Then by Baire's category theorem, $X$ is a set of second category. We have shown that $E^{\mathrm{C}}$ is a first category set a.s.. Thus $E$ is second category a.s.. It follows that $\dim_{\rm{P}} E=\dim_{\rm{P}} X=s$ a.s..	
\end{proof}

\section{dynamical covering problem}

Our results can be applied to the dynamical covering problem as follows.

We say that a metric measure preserving system (m.m.p.s. for short) $(X,\mathcal{B},\mu,T,d)$ is $\textit{exponentially mixing}$ if there exist two constants $c>0$ and $0<\gamma<1$ such that 
\[|\mu(E|T^{-n}F)-\mu(E)|\le c\gamma^n~~~~~(n\ge1)\]
holds for any ball $E$ and any measurable set $F\in\mathcal{B}$ with $\mu(F)>0$. Here $\mu(A|B)$ denotes $\frac{\mu(A\cap B)}{\mu(B)}$. Sometimes we say $\mu$ is exponentially mixing.

We assume that $(X,d)$ is compact, endowed with a Borel probability measure $\mu$ which is Ahlfors $s$-regular $(0<s<\infty)$. Define the $\textit{dynamical covering set}$ as 
\[E(x,l_n)=\bigl\{y\in X\colon d(T^n x,y)<l_n~\text{for~infinitely ~many}~n\ge0\bigr\},\]
where $x\in X$ and $\{l_n\}_{n\ge0}$ is a sequence of positive real numbers which is decreasing to zero. The dynamical covering set $E(x,l_n)=\bigcap_{k=0}^{\infty}\bigcup_{n=k}^{\infty}B(T^nx,l_n)$. We study the size of $E(x,l_n)$ from the viewpoint of measure, dimension and topology.

\begin{thm}\label{thmm}
Let $(X,\mathcal{B},\mu,T,d)$ be an exponentially mixing m.m.p.s. and the measure $\mu$ be Ahlfors s-regular.
	For $\mu$-almost all $x\in X$, we have
	\[\mu\bigl(E(x,l_n)\bigr)=
	\begin{cases}
	0&  \text{if $\sum_{n=0}^{\infty}l_n^s<\infty$}\\
	1&  \text{if $\sum_{n=0}^{\infty}l_n^s=\infty$}
	\end{cases}.\]
\end{thm}

\begin{thm}\label{thmm}
Let $(X,\mathcal{B},\mu,T,d)$ be an exponentially mixing m.m.p.s. and the measure $\mu$ be Ahlfors s-regular. Assume that $f$ is a doubling dimension function with $f(r)/r^s$ being nondecreasing as $r\to0$. Then for any ball $B$ of $X$,
\[\Hm^f(E(x,l_n)\cap B)=
\begin{cases}
	0&  \text{if $\sum_{n=0}^{\infty}f(l_n)<\infty$}\\
	\Hm^f(B)&  \text{if $\sum_{n=0}^{\infty}f(l_n)=\infty$}
	\end{cases} \quad \]
holds for $\mu$-almost all $x\in X$. Furthermore
	\[\dim_{\rm{H}} E(x,l_n)=\alpha ~a.e.,\]
	where $\alpha=\inf\{t\le s\colon \sum_{n=0}^{\infty}l_n^t<\infty\}$.
\end{thm}

\begin{thm}\label{thmd}
	Let $(X,\mathcal{B},\mu,T,d)$ be an exponentially mixing m.m.p.s. and $\mu$ be Ahlfors s-regular. Then the dynamical covering set $E(x,l_n)$ is dense in X a.e.. 
\end{thm}

\begin{cor}\label{cor5}
	Let $(X,\mathcal{B},\mu,T,d)$ be an exponentially mixing m.m.p.s. and $\mu$ be Ahlfors s-regular. For $\mu$-almost all $x\in X$, we have $\dim_B E(x,l_n)=s$.
\end{cor}

\begin{thm}\label{thmt}
	Let $(X,\mathcal{B},\mu,T,d)$ be an exponentially mixing m.m.p.s. and $\mu$ be Ahlfors s-regular. Therefore $E(x,l_n)$ is a residual set and moreover of second category for $\mu$-almost all $x\in X$. In particular, $\dim_P E(x,l_n)=s$ a.e..
\end{thm}

\begin{rem}\label{rem1}
         Now we give some systems which are exponentially mixing. 
         \begin{enumerate}[(i)]
        \item For the doubling map $Tx=2x$ $(\bmod\thinspace1)$ on the interval $[0,1)$, Gibbs measure $\mu$ associated to H\"older potentials is exponentially mixing. 
        \item For the $\beta$-shift $T_{\beta}x=\beta x$ $(\bmod\thinspace1)$ on the interval $[0,1)$, the Parry measure is exponentially mixing.
        \item  For the Gauss map $Sx=\{\frac{1}{x}\}$ $(\bmod\thinspace1)$ on the interval $[0,1)$, the Gauss measure is exponentially mixing. 
        \end{enumerate}
	
\end{rem}

\subsection*{Acknowledgements}
This work is  supported by  NSFC 11671151 and Guangdong Natural Science Foundation 2018B0303110005.

\end{document}